\documentclass{amsart}[12pt]
\usepackage{amsmath,amssymb,amsfonts,lscape, amsthm, amscd}

 \makeatletter \oddsidemargin.9375in
\evensidemargin \oddsidemargin \marginparwidth1.9375in
\makeatother
 \newtheorem{theorem}{Theorem}[section]
 \newtheorem{cor}[theorem]{Corollary}
 \newtheorem{lem}[theorem]{Lemma}

  \newtheorem{rem}[theorem]{Remark}
 \newtheorem{exam}[theorem]{Example}

 \numberwithin{equation}{section}

\begin{document}

\begin{abstract}
In this paper we  deal with the convergence of sequences of  positive linear maps to    a (not assumed to be linear)
 isometry on spaces of continuous  functions. We obtain generalizations of known Korovkin-type results  and provide several illustrative examples.
\end{abstract}

\title[theorems of Korovkin-type]{Korovkin-type results on convergence of sequences of positive linear maps on function spaces}

\author{Maliheh Hosseini}
\address{Department of Mathematics,
K. N. Toosi  University of Technology, Tehran, 16315-1618, Iran}
\email{m.hosseini@kntu.ac.ir}

\author{Juan J. Font}
\address{Departamento de Matem\'aticas\\
Universitat Jaume I\\ Campus Riu Sec\\ 8029 AP, Castell\'on,
Spain} \email{font@mat.uji.es}
\thanks{2010 {\em Mathematics Subject Classification}.
Primary 41A36; Secondary 46E15}
\thanks{Key words and phrases: Function space, Korovkin's theorem,  Choquet boundary,
 positive linear map.}

\thanks{J.J. Font is supported by Spanish Government grant MTM2016-77143-P  (AEI/FEDER, UE) and Generalitat Valenciana (Projecte GV/2018/110)}

\maketitle

\section{Introduction }

One of the most impressive results in approximation theory is, without doubt, Korovkin's theorem on convergence of positive linear operators on a space of continuous functions. More explicitly, Korovkin's theorem (often called Korovkin's first theorem) states that if a  sequence $\{T_n\}$ of positive linear maps   on $C_{\Bbb R}[0,1]$ converges to the identity operator on the quadratic polynomials, then $T_n f$ converges to $f$ for all $f\in C_{\Bbb R}[0,1]$ (\cite{Ko}). This result arose from a generalization of the well-known proof of Weierstrass's approximation theorem given by S. Bernstein. Its strength and simplicity have produced, as it is clearly imaginable,  a wide range of applications and generalizations. One of them deals with substituting the identity operator by other operators and the closed interval $[0,1]$ by other spaces. Others center on finding subsets of  function spaces, known  as
Korovkin sets or test functions, which guarantee that  the convergence of a sequence of positive linear maps holds on the whole space provided      it holds on them. For more details and other aspects of this topic, we refer to the monographs \cite{8,46}, the recent survey paper by Altomare \cite{Alt}, and the references therein.

Let  $X$ and $Y$ be compact Hausdorff
spaces, $M$ be a unital subspace of $C(X)$, and  $S$ be  a function space included in $M$. In  \cite{HO}, the authors studied the convergence of a sequence of unital linear contractions towards a fixed linear isometry. Indeed, they proved that, under certain assumptions, if each $T_n: M \longrightarrow C(Y)$ ($n\in \Bbb N$) is a unital linear contraction and $T_\infty: M \longrightarrow C(Y)$ is a linear isometry such that $\{T_n f\}$ converges to $T_\infty f$  for all $f\in S$, then  $\{T_n f\}$ converges to $T_\infty f$ for all $f\in M$, not only pointwise but also uniformly. In this paper we  deal with the convergence of sequences of (not necessarily  contractions) positive linear maps to    a (not assumed to be linear)
 isometry on spaces of continuous  functions by combining ideas given in   \cite{HO} and in the original proof of Korovkin's theorem. In particular, we obtain proper generalizations of
  \cite[Theorems 3.1 and 4.1]{HO}  and of several classical Korovkin-type results, and provide several illustrative examples.

\section {Preliminaries}
For any compact Hausdorff
space  $X$, let  $C(X)$ denote the space of
continuous real or  complex-valued functions on $X$, equipped with the uniform  norm $\|\cdot\|$. Note that we write $C_{\Bbb R}(X)$ instead of $C(X)$ when we want to consider only real-valued case. A unital
subspace  $S$ of $C(X)$ is called a {\it function space} on
$X$ if $S$  separates  the points of $X$ in the sense that
for each $x,x'\in X$ with $x\neq x'$ there exists a function
$f\in S$ such that  $f(x)\neq f(x')$.

Let $S$ be a subspace of $C(X)$, which we always assume to be linear. We denote by $\mathcal{B}_{S^{\ast}}$ the  closed unit ball of the dual
space of $(S, \|\cdot\|)$. A nonempty subset $E$ of
$X$ is called a \textit{boundary} for $S$ if each function in $S$
attains its maximum modulus within $E$.
The {\em Choquet boundary} $Ch(S)$ of $S$ is the non-empty set of all points
$x\in X$ for which $\delta_{x}$, the evaluation functional  at
$x$, is an extreme point of the closed  unit ball $\mathcal{B}_{S^{\ast}}$. Namely, we have $ext(\mathcal{B}_{S^{\ast}})=\Bbb T \,  Ch(S)=\{\alpha x: \alpha\in \Bbb T \, \textrm{and}  \, x\in Ch(S)\}$, where $\Bbb T=\{z\in \Bbb C: |z|= 1\}$. It is known that $Ch(S)$  is a boundary for  $S$. In particular, one can obtain the following remark immediately:

\begin{rem}  \label{RC}
\em{
%(1) It should be noted that there is no continuity assumption for $T_n$'s ($n\in \Bbb N$) and the linearity assumption  for $T_\infty$ in the above theorems.
If for each $x\in X$ there is a function $h\in S$ such that $h(x)=1$ and $|h(y)|<1$ for any $y\neq x$, then $Ch(S)=X$. For example, as in Korovkin's original theorem, if we assume $X=[0,1]$ and $S=Span\{1,x,x^2\}$, then $h(x):=1-(x-a)^2$, $a\in [0,1]$, yields $Ch(S)=[0,1]$.}
\end{rem}

In the sequel,  unless otherwise stated,
it is assumed that $X$ and $Y$ are     compact Hausdorff spaces, $M$ is a \textit{self-conjugate} subspace of $C(X)$  in the sense that $\bar{f}\in M$ whenever  $f\in M$,
and  $S$ is a function space included in $M$.

A linear map $T:M\longrightarrow C(Y)$  is called  \textit{positive} if $T f\geq 0$ holds for all $f\geq 0$.

%Note that when $M$ is a Banach lattice,  $T$ is necessarily continuous.

Let $f,f_1,f_2,...\in C(X)$ and  $X_0\subseteq X$. If  $\{f_n\}$ converges pointwise to $f$ on $X_0$,  we write $f_n\longrightarrow f$ on $X_0$. Also, we   omit $X_0$ when $X_0= X$.

Given $f,g\in C(X)$, we shall write $f \otimes 1+1 \otimes g$ to denote the function in $C(X\times X)$ such that $(f \otimes 1+1 \otimes g)(x,x'):=f(x)+g(x')$.
Furthermore, if $T,T':S\subseteq C(X)\longrightarrow C(Y)$, then we set $(T\otimes T1\, T')(f\otimes 1+1\otimes g)(y):=Tf(y)+T1(y) T'g(y)$ for all $f,g\in S$ and $y\in Y$.

Finally let us state the following lemma which is used in the proofs of our results.

\begin{lem} \label{B} \cite[Theorem 2.2.6]{Br} %(Strong version of Lemma \ref{ad})
Let $S$ be a  function space on $X$ and $x_0\in X$. Then  $x_0\in Ch(S)$
if and only if for any $\alpha,\beta\in (0,\infty)$ with $\alpha< \beta$ and any open neighborhood $U$ of $x_0$, there is a function $f\in S$ such that
\emph{Re}$f\leq 0$  on $X$,   \emph{Re}$f<-\beta$ on $U^{c}$ and \emph{Re}$f(x_0)>- \alpha$.
\end{lem}

\section{Results}

 \begin{theorem}  \label{t1}
 Suppose that  $\{T_n\}$ is a sequence of positive linear maps from $M$ into $C(Y)$, and $T_\infty$ is an isometry from $M$ onto a subspace
 $T_\infty(M)$ of $C(Y)$. % and $T=T_\infty(S)$.

(a) If $T_n f\longrightarrow T_\infty f$ for all $f\in S$, then  $T_n f\longrightarrow T_\infty f$  on $Ch(T_\infty(S))$ for all $f\in M$.

(b) Let %$N_n:=T_n(M)$ and $N:=\textrm{Span} \, \bigcup\limits_{1\leq n\leq \infty} N_n$.
$N:=\textrm{Span} \, \bigcup\limits_{1\leq n\leq \infty} T_n(M)$. If,  in part (a), $Ch(N) \subseteq Ch(T_\infty(S))$ and  the set $\{T_n 1:n\in \Bbb N\}$ is bounded, then $T_n f\longrightarrow T_\infty f$ for all $f\in M$.

% (c) $Ch(N)\subseteq Ch(T)$ and $\sup_{n\in \Bbb N} \|T_n\|< \infty$, then $T_n f\longrightarrow T_\infty f$ for all $f\in M$.
\end{theorem}

 \begin{proof}
We will base the proof of (a)  through the following steps.

\textbf{Step 1.} For each triple    of distinct points $x,x',z\in Ch(M)$, there exists a function $h\in M$ such that $|h(x)|\neq |h(x')|$ and $h(z)=0$.

Since $M$ is a self-conjugate function space we can find a real-valued function $f\in M$ such that $f(x)=1$ and $f(x')=0$.  Now we consider the following cases based on the value of $f$ at $z$:
\begin{itemize}
%\item $f(z)=0$. Clearly, the result is valid in this case.

\item $f(z)=1$. Clearly, $h=1-f$ is the desired function.

\item  $f(z)\neq 1, \frac{1}{2}$. Take $h=f-f(z)$.

\item $f(z)=\frac{1}{2}$. In this case we choose a non-negative  function $g$ in $M$
with $g(x),g(x')>3$ and $g(z)<\frac{1}{2}$, by Lemma \ref{B}. If $g(x')-g(x)=2$, then $h=g-g(z)$ is the desired function. Otherwise, we can see that $h=2f+g-g(z)-1$ satisfy the requested properties.
\end{itemize}
\medskip
\textbf{Step 2.} $T_\infty$ is a linear isometry.

Note that $T_\infty 0= \lim T_n 0=0$. Then  according to the Mazur-Ulam theorem \cite{M}, $T_\infty$ is a real-linear isometry. Hence now we only need to consider the complex case. Let us point out  that $T_\infty 1= \lim T_n 1\geq 0$. Taking into account
Step 1, from \cite[Theorem 2.3]{Miu} it follows that $T_\infty 1= 1$ and there exist a (possibly  empty) clopen subset $K$ of $Ch(T_\infty(M))$, and a continuous surjective map
 $\phi:Ch(T_\infty(M))\longrightarrow
Ch(M)$ such that for all $f\in M$,
\[T_\infty f=\left\lbrace
  \begin{array}{c l}
    f\circ \phi &  \,\, \textrm{on} \,  K,\\
  \overline{f\circ \phi} &  \,\, \textrm{on} \,
Ch(T_\infty(M))\setminus K.
  \end{array}
\right. \]
But $T_\infty i= \lim T_n i=i \lim T_n 1=iT_\infty 1=i$, which implies that $K=Ch(T_\infty(M))$. Hence taking into account that $Ch(T_\infty(M))$ is a boundary for $T_\infty(M)$, we deduce that
$T_\infty$ is a linear isometry.

\medskip
\textbf{Step 3.} For each $f\in M$, $T_nf \longrightarrow T_\infty f$ on $Ch(T_\infty(S))$.

%\textbf{The proof of this step is a minor modification of \cite{HO}.}

By \cite[Lemma 2.5]{HO} (or \cite[Corollary 3.2]{AF}), there is a continuous surjection
 $\varphi:Ch(T_\infty(S))\longrightarrow
Ch(S)$ such that
\[T_\infty f(y)=f( \varphi(y)) \,\,\, (f\in S, y\in Ch(T_\infty(S))).\]

Let $f\in M$ and $\epsilon>0$. Then we can define a function in $C(X\times X)$ as $F:=f \otimes 1-1 \otimes f$. Clearly, $F=0$ on the subset  $\bigtriangleup_X=\{(x,x):x\in X\}$ of $X \times X$. Then
there is an open neighborhood $U$ of  $\bigtriangleup_X$ with $|F|<\epsilon$ on $U$.

Let $y'\in Ch(T_\infty(S))$ and $x'=\varphi(y')$. Choose an open neighborhood $V_{x'}$ of $x'$ such that $V_{x'}\times V_{x'}\subseteq U$. By Lemma \ref{B}, we find a function
$f_{y'}\in S$ such that
\[\textrm{Re} f_{y'}\geq 0 \,\, \textrm{on}\, \,X, \, \, \textrm{Re} f_{y'}\geq 1\, \,\textrm{ on} \,\,V_{x'}^{c},\, \,\textrm{Re} f_{y'}(x')< \epsilon. \]
Put $F_{y'}=f_{y'} \otimes 1+ 1 \otimes f_{y'}$. It is clear that $\textrm{Re} F_{y'}\geq 0$ on  $X \times X$ and $\textrm{Re} F_{y'}\geq 1$ on $U^{c}$. Hence we have
\[\textrm{Re} F\leq \|F\|\leq \|F\| \textrm{Re} F_{y'}\,\,\, \textrm{on} \,\, U^{c},\]
which yields
$ |\textrm{Re} F|\leq 1 \otimes \epsilon + \|F\| \textrm{Re} F_{y'}$  on $X \times X$.
In other words,
\[ -(1 \otimes \epsilon + \|F\| \textrm{Re} F_{y'}) \leq \textrm{Re} F \leq 1 \otimes \epsilon + \|F\| \textrm{Re} F_{y'}\,\,\, \textrm{on} \,\, X \times X.\]
Hence for each $y\in X$ we get
\[-\epsilon-2\|F\| \textrm{Re} f_{y'}-\|F\| \textrm{Re} f_{y'}(y)+\textrm{Re} f(y)\leq \textrm{Re} f-\|F\| \textrm{Re} f_{y'}\leq \epsilon+\|F\| \textrm{Re} f_{y'}(y)+\textrm{Re} f(y).\]

Since $\{T_n\}$ is a sequence of linear positive maps, it follows that
\[-  2\|F\| T_n (\textrm{Re} f_{y'})+(-\epsilon-\|F\| \textrm{Re} f_{y'}(y)+\textrm{Re} f(y)) T_n 1\leq T_n( \textrm{Re} f)-  \|F\| T_n (\textrm{Re} f_{y'})\leq\]\[ T_n1 (\epsilon+\|F\| \textrm{Re} f_{y'}(y)+\textrm{Re} f(y))\]
for each $y\in X$. Now, from the  representation of  $T_\infty$ on $M$ (Step 2), we deduce that
\[-  2\|F\| T_n( \textrm{Re} f_{y'})(z)+T_\infty(-\epsilon-\|F\| \textrm{Re} f_{y'}+\textrm{Re} f) (z')T_n 1(z)\leq T_n (\textrm{Re} f)(z)-\|F\|  T_n (\textrm{Re} f_{y'})(z)\leq\]
\[ T_n1 (z) T_\infty(\epsilon+\|F\| \textrm{Re} f_{y'}+\textrm{Re} f)(z')\]
for any $z\in Y$ and $z'\in Ch(T_\infty(M))$. Thus, again since $T_\infty 1= 1$, $T_\infty$ is   a positive linear map and also $Ch(T_\infty(M))$ is a boundary for $T_\infty(M)$, it is observed that the above relation holds for all $z,z'\in Y$. Therefore, especially  we get
\[-||F|| T_n( \textrm{Re} f_{y'}) - T_n 1 T_\infty (\epsilon+\|F\| \textrm{Re} f_{y'}) \leq T_n (\textrm{Re} f)-  T_n 1  T_\infty (\textrm{Re} f)  \leq T_n1 T_\infty(\epsilon+\|F\| \textrm{Re} f_{y'})+||F|| T_n (\textrm{Re} f_{y'})\]
on $Y$. Rewriting the above inequality adopted to our notation in Section 2 we have
\[- (T_n\otimes T_n 1 T_\infty) (1 \otimes \epsilon + \|F\| \textrm{Re} F_{y'})\leq (T_n\otimes T_n 1 T_\infty) (\textrm{Re} F)  \leq (T_n\otimes T_n 1 T_\infty)(1 \otimes \epsilon + \|F\| \textrm{Re} F_{y'}),\]
equivalently,
\[|(T_n\otimes T_n 1  T_\infty)(\textrm{Re} F) |\leq  (T_n\otimes T_n 1 T_\infty)(1 \otimes \epsilon + \|F\| \textrm{Re} F_{y'}).\]

Consequently,   from the fact that each $T_n$ is a positive linear map and the representation of  $T_\infty$, it follows that
\begin{align*}\frac{}{}
|\textrm{Re}(T_n\otimes T_n1 T_\infty)(F)|&= |\textrm{Re} T_n f- \textrm{Re} (T_n1 T_\infty f)| \\ &=
|T_n (\textrm{Re} f)-T_n1 T_\infty (\textrm{Re} f)|  \\ &=
|(T_n\otimes T_n1 T_\infty)(\textrm{Re}F)|\\ &\leq
(T_n\otimes T_n1 T_\infty)(1 \otimes \epsilon + \|F\| \textrm{Re} F_{y'})\\ &=(T_n\otimes T_n1 T_\infty)(1 \otimes \epsilon) +(T_n\otimes T_n1 T_\infty)( \|F\| \textrm{Re} F_{y'})\\ &=
 \epsilon T_n1 + \|F\|(T_n(\textrm{Re} f_{y'})+ T_n1 T_\infty(\textrm{Re} f_{y'}))\\ &=\epsilon T_n1 + \|F\|(\textrm{Re} T_n f_{y'}+ T_n1 \textrm{Re} T_\infty f_{y'})\\ &\leq \epsilon T_n1 + \|F\|(| T_n f_{y'}- T_\infty f_{y'}|+  T_n1 \textrm{Re} T_\infty f_{y'}+ \textrm{Re} T_\infty f_{y'}),
\end{align*}
which is to say,
\[|\textrm{Re}(T_n\otimes T_n1 T_\infty)(F)|\leq \epsilon T_n1 + \|F\|(| T_n f_{y'}- T_\infty f_{y'}|+  T_n1 \textrm{Re} T_\infty f_{y'}+ \textrm{Re} T_\infty f_{y'}).\]
Thus, from the latter inequality, the  representation of  $T_\infty$ and for any sufficiently large integer $n$, we get
\begin{align*}\frac{}{}
|\textrm{Re} T_n f(y')-\textrm{Re} T_\infty f(y')| &\leq  |\textrm{Re} T_n f(y')- T_n1(y') \textrm{Re} T_\infty f(y')|  + |T_n1(y') \textrm{Re} T_\infty f(y')- \textrm{Re} T_\infty f(y')|\\ &\leq \epsilon T_n1(y')+ \|F\|(| T_n f_{y'}(y')- T_\infty f_{y'}(y')| +  T_n1(y') \textrm{Re}  f_{y'}(x')+\\ & \textrm{Re}  f_{y'}(x'))+ |\textrm{Re} T_\infty f(y')|| T_n1(y')-1|
 \\ &\leq 2 \epsilon+ \|F\|(\epsilon + 2 \epsilon+  \epsilon)+ \|f\| \epsilon\\ & =(2+4 \|F\|+ \|f\|)\epsilon.
\end{align*}
%which shows that for any sufficiently large integer n,
%\begin{align*}
%|\textrm{Re} T_n f(y')-\textrm{Re} T_\infty f(y')| \leq & \epsilon+ \|F\|(| T_n f_{y'}(y') -T_\infty f_{y'}(y') |+ 2 \textrm{Re}  f_{y'}(x' ))\\ &< \epsilon+ \|F\|(\epsilon +2 \epsilon )=(1+3 \|F\|) \epsilon.
%\end{align*}
Hence $\textrm{Re} T_n f\longrightarrow \textrm{Re} T_\infty f$ on $Ch(T_\infty(S))$. By replacing $f$ by $-if$, we see that  $\textrm{Im}T_n f\longrightarrow \textrm{Im} T_\infty f$ on $Ch(T_\infty(S))$. Therefore,  $T_n f\longrightarrow T_\infty f$ on $Ch(T_\infty(S))$, which completes the proof of part (a).

\medskip
(b) % Under the assumption "the set $\{T_n f:n\in \Bbb N\}$ is bounded for all $f\in S$", we claim that for each $f\in M$, $T_n f\longrightarrow T_\infty f$.
We first claim that $\|T_n\|\leq \sqrt{2} \|T_n 1\|$, where $\|T_n\|$ is the operator norm of $T_n$ (for each $n\in \Bbb N$). To see this, assume that  $g\in M$  is real-valued and has supremum norm at most 1.  Then
$-1\leq g \leq 1$ and thus, $-T_n 1\leq T_n g \leq T_n 1$, which implies that $\|T_n g\|\leq  \|T_n 1\|$. In the real case, this shows that $T_n$ is continuous and the claim holds.
In the complex case,  from this argument and the fact that $M$ is
self-conjugate, it easily follows that $\|T_n\|\leq \sqrt{2} \|T_n 1\|$.

Let $f\in M$. Taking into account  the above claim and the boundedness of  $\{T_n 1:n\in \Bbb N\}$, we deduce that  the set $\{T_n f:n\in \Bbb N\}$ is bounded.  %From the arguments given in part (a) it follows that for each $n\in \Bbb N$,
%\[|\textrm{Re} T_n f-T_n 1 \textrm{Re} T_\infty f| \leq \epsilon T_n1 + \|F\|(\textrm{Re} T_n f_{y'}+ T_n1 \textrm{Re} T_\infty f_{y'})\]
%which implies that
%\begin{align*}
%\|\textrm{Re} T_n f\| &\leq  \epsilon\|T_n 1\| + \|F\| (\textrm{Re} T_n f_{y'}+ \|T_n 1\| \textrm{Re} T_\infty f_{y'})+ \|T_n 1\| \|\textrm{Re} T_\infty f\|\\ &\leq \eta \epsilon+ \|F\| (\sup_{i\in \Bbb N} \|T_i f_{y'}\|+ \eta \|f_{y'}\|)+ \eta \|f\|,
%\end{align*}
%where $\eta=\sup_{i\in \Bbb N} \|T_i 1\|$.
%Hence for each $n\in \Bbb N$, $\|\textrm{Re} T_n f\| \leq t$, where $t=\eta \epsilon+ \|F\| (\sup\limits_{i\in \Bbb N} \|T_i f_{y'}\|+ \eta \|f_{y'}\|)+ \eta \|f\|$. Similarly there is a scalar  $t'>0$ such that  $\|\textrm{Im} T_n f\| \leq t'$ for each $n\in \Bbb N$. Consequently, the set  $\{T_n f:n\in \Bbb N\}$ is bounded.
Now one can follow the last part of the proof of
 \cite[Theorem 3.3]{HO} to conclude that $T_n f\longrightarrow T_\infty f$ on $Y$ and we include it for completeness. Assume  that $\sim$
is the equivalence relation on $Y$ defined by
\[y\sim y'\Leftrightarrow g(y)=g(y') \,\, \hspace{1cm} \forall g\in N.\]
The quotient space of $Y$ by $\sim$ is denoted by $Y/\sim$, and $\hat{y}$ will stand for the image of $y\in Y$ under the canonical map $\hat{\cdot}$ from Y onto $Y/\sim$. Moreover, we define $\hat{g}(\hat{y})=g(y)$ for all $g$ in $N$ and $y$ in $\hat{Y}=\{\hat{y}:y\in Y\}$.  It is apparent that $\hat{N}=\{\hat{g}:g\in N\}$ is a function space on  the compact space $\hat{Y}$.

By \cite[Section V]{Bi} and  \cite[Section 4]{P1}, for any $y\in Y$, there exits a positive measure $\mu$ on the $\sigma$-ring of subsets of $\mathcal{B}_{\hat{N}^{\ast}}$ generated
by  $ext(\mathcal{B}_{\hat{N}^{\ast}})$ and the Baire subsets of $\mathcal{B}_{\hat{N}^{\ast}}$ which represents $\hat{y}$ and  $\mu(\mathcal{B}_{\hat{N}^{\ast}})=1$.
From part (a), it is clear that $\widehat{T_n f}\longrightarrow \widehat{T_\infty f}$ on $\widehat{Ch(T_\infty(S))}$. Hence, since
$ext(\mathcal{B}_{\hat{N}^{\ast}})=\Bbb T \, Ch(\hat{N})\subseteq \Bbb T \, \widehat{Ch(T_\infty(S))}$ and the set $\{T_n f: n\in \Bbb N\}$ is bounded, from the Lebesgue's dominated convergence theorem we get
\[T_n f(y)=\widehat{T_n f}(\hat{y})=\int_{\mathcal{B}_{\hat{N}^{\ast}}} \widehat{T_n f}\longrightarrow \int_{\mathcal{B}_{\hat{N}^{\ast}}} \widehat{T_\infty f} d\mu=\widehat{T_\infty f}(\hat{y})=T_\infty f(y).\]
Therefore, $T_n f\longrightarrow T_\infty f$, as desired. \end{proof} %\textbf{ In fact, I added the assumption "the set $\{T_n f:n\in \Bbb N\}$ is bounded for all $f\in S$" in part (b), just to be able to use the  Lebesgue's dominated convergence theorem. It would be good if we could obtain the result without assuming any additional  hypotheses. I need your opinion.} % In particular, if $\sup_{n\in \Bbb N} \|T_n\|< \infty$, then for each $f\in M$, the set $\{T_n f:n\in \Bbb N\}$ is bounded and so the result is obtained from the above argument.

Let us recall here the famous Arzela-Ascoli theorem, which will be used in the proof of the next result.

\noindent
\textbf{Theorem (Arzela-Ascoli).} Given a subset $A$ of  $C(X)$, the following statements are equivalent:

 (1) $A$ is a compact subset of $(C(X), \|\cdot\|)$.

(2) $A$ is closed,  bounded, and equicontinuous in the sense that for each $x\in X$ and $\epsilon>0$, there exists a neighborhood $V$ of $x$ such that $|f(y)-f(x)|<\epsilon$ for all $f\in A$ and $y\in V$.

 \begin{theorem}  \label{t2}
Let $\{T_n\}$ be a sequence of positive linear maps from $M$ into $C(Y)$, and $T_\infty$ be an isometry from $M$ onto a subspace
 $T_\infty(M)$ of $C(Y)$. %Let $M$, $S$, $\{T_n\}$, $N_n$, $N$ $T$ be as in Theorem 1.

 (a) If $\{T_n f\}$ converges uniformly to $T_\infty f$  for all $f\in S$, then  $\{T_n f\}$ converges uniformly to
 $T_\infty      f$  on each compact subset of $Ch(T_\infty(S))$ for all $f\in M$.

 (b) If, furthermore,  either $Ch(T_\infty(S))$ or $Ch(N)$ is compact and $Ch(N)\subseteq Ch(T_\infty(S))$, then  $\{T_n f\}$ converges uniformly to
 $T_\infty      f$ for any $f\in M$, where  $N$ is   as in Theorem \ref{t1}.
\end{theorem}

 \begin{proof}
%The proof is a modification of the proof of  \cite[Theorem 4.1]{HO}.
(a) As in the proof of Theorem \ref{t1}, there is  a continuous surjection
 $\varphi:Ch(T_\infty(S))\longrightarrow
Ch(S)$ such that for all $f\in S$,
\[T_\infty f(y)=f(\varphi(y)) \,\,\, (f\in S, y\in Ch(T_\infty(S))).\]
Suppose that $K$ is a compact subset of $Ch(T_\infty(S))$. Let $f\in M$,  $y'\in K$ and $\epsilon>0$. Put $F=f \otimes 1-1 \otimes f$   and  $x'=\varphi(y')$. As before, we choose an open neighborhood
 $V_{x'}$ of $x'$ and a function $f_{y'}\in S$ such that
 $\textrm{Re} f_{y'}\geq 0$  on $X$,  $\textrm{Re} f_{y'}\geq 1$ on $V_{x'}^{c}$ and $\textrm{Re} f_{y'}(x')< \epsilon$,
and we also  have
\[|\textrm{Re} T_n f-\textrm{Re} T_\infty f| \leq \epsilon T_n 1 + \|F\|(| T_n f_{y'} -T_\infty f_{y'} |+  \textrm{Re} T_\infty  f_{y'}+ T_n 1 \textrm{Re} T_\infty  f_{y'})+ |\textrm{Re} T_\infty f|| T_n1-1|,\]

%============================================
%
%
%Clearly, $K\subseteq \bigcup_{y'\in K} W_{y'}$, where $W_{y'}=\{y\in K: \textrm{Re} T_\infty f_{y'} (y)<\epsilon\}$. Since $K$ is compact, there is a finite number of points $y_1',..., y_p'$ in $K$ such that $K\subseteq \bigcup_{i=1}^{p} W_{y'_i}$. Hence for any $y\in K$,  $y\in W_{y'_k}$ for some $k\in \{1,..., p\}$ and we get
%\begin{align*}
%|\textrm{Re} T_n f(y)-\textrm{Re} T_\infty f(y)|&\leq \epsilon+ \|F\|(| T_n f_{y_k'}(y)-T_\infty f_{y_k'}(y)|+ 2 \textrm{Re} T_\infty f_{y_k'}(y))
%\\ &\leq \epsilon+ \|F\|(\| T_n f_{y_k'}- T_\infty f_{y_k'}\| +2 \epsilon)\\ &\leq \|F\| \max_{1\leq i\leq p} \| T_n f_{y_i'}- T_\infty f_{y_i'}\|+(1+2\|F\|)\epsilon,
%\end{align*}
%Applying the above discussion to $-if$ instead of $f$, we can find functions $g_{y_1''},..., g_{y_m''}\in S$ such that for any $y\in K$,
%\[|\textrm{Im} T_nf(y)-\textrm{Im} T_\infty f(y)|\leq \|F\|  \max_{1\leq i\leq m} \| T_n g_{y_i''}- T_\infty g_{y_i''}\|+(1+2\|F\|)\epsilon.\]
%Hence for any sufficiently large integer $n$,
%\begin{align*}
%|T_n f(y)-T_\infty f (y)|&\leq |F\| (\max_{1\leq i\leq p} \| T_n( f_{y_i'})- T_\infty( f_{y_i'})\|+\max_{1\leq i\leq m}  \| T_n( g_{y_i''})- T_\infty( g_{y_i''})\|)+2(1+2\|F\|)\epsilon
%\\ &<
%(2+2\|F\|+4 \|F\|)\epsilon=(2+6 \|F\|)\epsilon,
%\end{align*}
%which yields that $\{T_n f\}$ converges uniformly to
% $T_\infty      f$  on $K$.
%
%================================================
on $Y$. Now, we prove the following claim.

\textbf{Claim:} The set  $\{T_n f:n\in \Bbb N\}$ is equicontinuous at $y'$.

Since  $\{T_n f_{y'}\}$ and  $\{T_n 1\}$  converge uniformly to
 $T_\infty      f_{y'}$ and 1, respectively,  there is an integer $n_0$ such that for each $n\geq n_0$, $\|T_n f_{y'}-T_\infty f_{y'}\|< \epsilon$ and $\|T_n 1-1\|< \epsilon$. On the other hand, $\textrm{Re} T_\infty f_{y'}(y')< \epsilon$
and so, from the continuity of $\textrm{Re} T_\infty f_{y'}$ and $T_\infty f$, we can choose a neighborhood $W_{y'}$ of $y'$ so that the inequalities $\textrm{Re} T_\infty f_{y'}<\epsilon $ %$\textrm{Re} T_\infty f_{{y'}_{|W_{y'}}}<\epsilon $
and $|T_\infty f-T_\infty f (y')|<\epsilon$ hold on $W_{y'}$. Hence, letting $\eta=\sup_{i\in \Bbb N} \|T_i 1\|$,  for each $y\in W_{y'}$ and $n\geq n_0$ we get
\begin{align*}
|\textrm{Re} T_n f(y)-\textrm{Re} T_n f(y')|&\leq |\textrm{Re} T_n  f(y)- \textrm{Re} T_\infty  f(y)|  +|\textrm{Re} T_n  f(y')- \textrm{Re} T_\infty  f(y')| +\\ &|\textrm{Re} T_\infty  f(y)- \textrm{Re} T_\infty  f(y')|\leq  \eta\epsilon+ \|F\|(| T_n f_{y'}(y)-T_\infty f_{y'}(y)|+  \textrm{Re} T_\infty f_{y'}(y)+  \\ &\eta T_\infty f_{y'}(y)) + \|f\| | T_n1(y)-1|+
 \eta\epsilon+ \|F\|(| T_n f_{y'}(y')-T_\infty f_{y'}(y')|+\\ &\textrm{Re} T_\infty f_{y'}(y')+ \eta T_\infty f_{y'}(y'))+\|f\| | T_n1(y')-1|+ |\textrm{Re} T_\infty  f(y)- \textrm{Re} T_\infty  f(y')|\\ &\leq  \eta\epsilon+ \|F\|(\epsilon+ \epsilon+ \eta\epsilon)+\|f\| \epsilon+ \eta\epsilon+\|F\|(\epsilon+ \epsilon+ \eta\epsilon)+\|f\| \epsilon+\epsilon\\ &=\epsilon(2\eta + 2\|f\|+ 4 \|F\|+2\eta \|F\|)+ \epsilon.
\end{align*}
Now, from the continuity of $T_1 f, ..., T_{n_0} f$, it follows that the set  $\{\textrm{Re} T_n f:n\in \Bbb N\}$ is equicontinuous at $y'$. Similarly, the set $\{\textrm{Im} T_n f:n\in \Bbb N\}$ is equicontinuous at $y'$, and, as a consequence, $\{T_n f:n\in \Bbb N\}$ is equicontinuous at $y'$, as claimed.

\medskip
Moreover, as observed in the proof of Theorem \ref{t1}(b), $\{T_n f:n\in \Bbb N\}$ is bounded. Therefore, from the Arzela-Ascoli theorem and Theorem \ref{t1}(a), it follows that each
subsequence $\{T_n f\}$ has a uniformly convergent sequence to $T_\infty      f$ on $K$. This argument shows that
$\{T_n f\}$ converges uniformly to
 $T_\infty      f$ on the  compact   set $K$.

\medskip
(b) When either $Ch(T_\infty(S))$ or $Ch(N)$ is compact, then, from the above discussion, we deduce that  $\{T_n f\}$ converges uniformly to
 $T_\infty      f$ on $Ch(N)$. Next, since  $Ch(N)$ is a boundary for $N$, it is immediately seen that $\{T_n f\}$ converges uniformly to
 $T_\infty      f$ (on $Y$).
 \end{proof}
  %When either $Ch(T)$ or $Ch(N)$ is compact, then from the fact that  $Ch(N)$ (and so $Ch(T)$) is a boundary for the set  $\{T_n f:n\in \Bbb N\}$ and the above discussion, it
 % follows that   $\|T_n f- T_m f\|_{\infty}\longrightarrow 0$ and since $T_n f\longrightarrow T_\infty f$, by Theorem \ref{t1}(b), we deduce that $\{T_n f\}$ converges uniformly to
% $T_\infty      f$.

%\textbf{Another proof instead of the above  part between ===== and ======, which is a little different from the proof of  \cite[Theorem 4.1]{HO}:}

\begin{rem}
\em{ We would like to remark that the sequential version of Korovkin's theorem does not yield its net version (see \cite{Sche}). However, it can be easily checked that our techniques hold  true when we replace the sequence  $\{T_n\}$ by a net of positive linear maps.}

%(1) It should be noted that there is no continuity assumption for $T_n$'s ($n\in \Bbb N$) and the linearity assumption  for $T_\infty$ in the above theorems.

%(1) In Theorem \ref{t1}(a), if $Ch(N)\subseteq Ch(T_\infty(S))$ and $\sup\limits_{n\in \Bbb N} \|T_n\|_o< \infty$ ($\|T_n\|_o$ is the operator norm of $T_n$), then  for each $f\in M$, the set $\{T_n f:n\in \Bbb N\}$ is bounded and so $T_n f\longrightarrow T_\infty f$.

\end{rem}

 In the following corollary, we obtain the main results of \cite{HO}, namely, \cite[Theorem 3.3]{HO} and \cite[Theorem 4.1]{HO} as consequences of Theorems \ref{t1} and \ref{t2}.

 \begin{cor}  \label{co}
 Let $M$ be a subspace of $C(X)$, $S\subseteq M$ be  a function space, $\{T_n\}$ be a sequence of unital linear contractions from $M$ into $C(Y)$,  $T_\infty$ be a linear isometry from $M$ into $C(Y)$,   and $Ch(N)\subseteq Ch(T_\infty(S))$, where  $N:=\textrm{Span} \, \bigcup\limits_{1\leq n\leq \infty} T_n(M)$. %$N$ is as in Theorem \ref{t1}.

 (a) If  $T_n f\longrightarrow T_\infty f$ for all $f\in S$, then $T_n f\longrightarrow T_\infty f$ for all $f\in M$.

 (b)  If  $\{T_n f\}$ converges uniformly to
 $T_\infty      f$ for all $f\in S$, then    $\{T_n f\}$ converges uniformly to
 $T_\infty      f$  on each compact subset of $Ch(T_\infty(S))$ for any $f\in M$. If, furthermore,  $Ch(T_\infty(S))$ or $Ch(N)$ is compact, then  $\{T_n f\}$ converges uniformly to
 $T_\infty      f$ for all $f\in M$.

\end{cor}

 \begin{proof}
In the context of real-valued function spaces, since every linear map $\mathcal{T}$ with $\|\mathcal{T}\|=\mathcal{T}(1)=1$ is positive (\cite{P}), the result follows immediately from Theorems  \ref{t1} and  \ref{t2}. Now let us consider the complex case. We note that
\[M+\overline{M}=\{f+\overline{g}: f,g\in M\}\]
is a self-conjugate subspace of $C(X)$. According to \cite[Lemma 2.5]{HO} (or \cite[Corollary 3.2]{AF}), there is a continuous surjection $\varphi:Ch(T_\infty(M))\longrightarrow
Ch(M)$ such that
\[T_\infty f(y)=f( \varphi (y))\,\,\, (f\in M, y\in Ch(T_\infty(M))).\]
Since $Ch(T_\infty(M)+\overline{T_\infty(M)})=Ch(T_\infty(M))$ and $Ch(M+\overline{M})=Ch(M)$ (\cite[Lemma 2.3]{HO})  are boundaries, $T_\infty$ can be extended to a linear
 isometry $\widetilde{T}_\infty:M +\overline{M}\longrightarrow C(Y)$ such that
 \[\widetilde{T}_\infty(f+\overline{g})(y)=f( \varphi (y))+ \overline{g( \varphi (y))}\,\,\, (f,g\in M, y\in Ch(T_\infty(M))).\]
 Moreover, by  \cite[Lemma 3.2]{HO}, each $T_n$ can be extended to a positive linear map $\widetilde{T}_n$ from $\overline{M}+M$ into $C(Y)$. Now,  we get the result from Theorems  \ref{t1} and  \ref{t2}.
\end{proof}

\section{Examples}

In this section  we provide several examples which show how our results can be applied.

\begin{exam} \label{example}
\em{ Let   $k\in \Bbb N\cup \{0,\infty\}$ and $C^{(k)}(I)$ denote  the   space of  $k$-times continuously
differentiable functions  on the interval
 $I=[0,1]$ which is a self-conjugate space.
Suppose that  $\{T_n\}$ is a sequence of positive linear maps from $C^{(k)}(I)$ into $C(I)$ satisfying
\[T_n 1\longrightarrow 1, \,\,\, T_n x\longrightarrow x, T_n x^{2}\longrightarrow x^{2}.\]
For each $a\in I$, the function $h(x)=1-(x-a)^{2}$ belongs to the function space $S=Span\{1,x,x^{2}\}$. Since  $h(a)=1$ and $|h(y)|<1$ for any $y\neq a$, we infer $Ch(S)=I$, by Remark \ref{RC}. Now from Theorem \ref{t1}, we conclude that
 $T_n f\longrightarrow f$ for all $f\in C^{(k)}(I)$.
Meantime, by Theorem  \ref{t2}, the same result holds true for  "uniformly convergence" instead of "pointwise  convergence", which can be also obtained from Korovkin's first theorem.}

%and, as a consequence, we obtain Korovkin's first theorem.}
\end{exam}

%It is worth mentioning that in \cite{Mu} the authors provided a sequence of linear maps $\{T_n:C(I)\longrightarrow C(I):n\in \Bbb N\}$ which are  positive on $C^{(2)}(I)$ but
%not on $C(I)$. Besides, one can see that  $T_n$ is not a unital contraction on   $C^{(2)}(I)$ ($n\in \Bbb N$). So, we can apply our results to the convergence of this sequence on $C^{(2)}(I)$ but not the ones in \cite{HO}.

%\begin{exam} \label{example}
%\em{Suppose that   $\{T_n:C(I)\longrightarrow C(I):n\in \Bbb N\}$ ($I=[0,1]$) is a sequence of positive linear maps such that
%\[T_n 1\longrightarrow 1, \,\,\, T_n x\longrightarrow \sqrt{x}, \,\,\, T_n x^{2}\longrightarrow x.\]
%Then, by Theorem  \ref{t1},
%\[T_n f(x)\longrightarrow T_\infty f(x)\,\,\,\,\, (f\in C(I), x\in I),\]
%where $T_\infty :C(I)\longrightarrow C(I)$ is the isometry defined as
%\[T_\infty f(x)=f(\sqrt{x})\,\,\,\,\, (f\in C(I), x\in I).\]
%We also note that, if $S=Span\{1,x,x^{2}\}$, then $I=Ch(T_\infty(S))$ by Remark \ref{RC} since, for each $a\in I$, we can take the function $h(x)=1-(\sqrt{x}-\sqrt{a})^{2}$.
%The same result holds true for  "uniformly convergence" instead of "pointwise  convergence" thanks to Theorem  \ref{t2}.}
%\end{exam}

\begin{exam} \label{example}
\em{ Let   $\Omega$ be a non-empty open subset of $\Bbb R^{p}$ and $K$ be a compact subset of  $\Omega$. The term \textit{multi-index} denotes an ordered $p$-tuple $\alpha=(\alpha_1,..., \alpha_p)$
of nonnegative integers $\alpha_i$. For each multi-index $\alpha$, consider the differential operator
\[D^{\alpha}=\left(\frac{\partial}{\partial x_1}\right)^{\alpha_1} ...\left(\frac{\partial}{\partial x_p}\right)^{\alpha_p},\]
if $\alpha\neq 0$, and $D^{\alpha}f=f$ if $\alpha= 0$. A function $f$ on $\Omega$ is said to belong to $C^{\infty}(\Omega)$ if $D^{\alpha}f\in C(\Omega)$  for all  multi-index $\alpha$. By $\mathcal{D}_K$ we denote the space $\{f|_K: f\in C^{\infty}(\Omega)\}$. Since $\mathcal{D}_K$ may be considered as a function space on $K$, from our results we deduce the following.

If $\{T_n:\mathcal{D}_K\longrightarrow C(K):n\in \Bbb N\}$ is a sequence of positive linear maps such that  $T_n 1\longrightarrow 1$,  $T_n (P_k)\longrightarrow P_k$,
$T_n (\sum_{k=1}^{p} P_k^{2})\longrightarrow \sum_{k=1}^{p} P_k^{2}$, where $P_k$ is the projection
\[P_k(x)=x_k \,\, \textrm{for} \,\, x=(x_1,...,x_p),\]
then  $T_n f\longrightarrow f$ for all $f\in \mathcal{D}_K$. A similar result holds true for  "uniformly convergence" instead of "pointwise  convergence".

Let us remark that for any $a=(a_1,...,a_p)\in K$, the function
\[h(x)=b_1-(P_1(x)-a_1)^{2}+...+b_k-(P_p(x)-a_p)^{2}\,\,\,\,\, (x=(x_1,...,x_p)\in \Omega),\] where
$b_i> \max\{|P_i(x)-a_i|: x\in K\}$, $i=1,..., p$,
implies that $a$ belongs to the Choquet boundary of $S=Span\{1,P_1, ..., P_p, P_1^{2}, ..., P_p^{2}\}$ by Remark \ref{RC}.
}
\end{exam}

The following example includes %some known results especially
 the complex Korovkin theorem.

\begin{exam} \label{example}
\em{ If %$\Bbb T=\{z\in \Bbb C: |z|= 1\}$ and
$\{T_n:C(\Bbb T)\longrightarrow C(\Bbb T):n\in \Bbb N\}$ is a sequence of positive linear maps such that  $T_n 1\longrightarrow 1$ and  $T_n z\longrightarrow z$,
then  $T_n f\longrightarrow f$ for all $f\in C(\Bbb T)$.
Notice that here if $z_0\in \Bbb T$, then the function $h(z)=\frac{z+z_0}{2}$ works for Remark \ref{RC} ($S=Span\{1,z\}$).

Let $D$ be the closed unit disc $\{z\in \Bbb C: |z|\leq 1\}$ and $\{T_n:C(D)\longrightarrow C(D):n\in \Bbb N\}$ be  a sequence of positive linear maps such that  $T_n 1\longrightarrow 1$,  $T_n z\longrightarrow z$,
$T_n |z|^{2}\longrightarrow |z|^{2}$,
then  $T_n f\longrightarrow f$ for all $f\in C(D)$.

It should be noted that since $T_n$ is positive, it is easily seen that $T_n \bar{z}=\overline{T_n z}$, which yields $T_n \bar{z}\longrightarrow \bar{z}$.  Hence  for each  $z_0\in D$, the function $h(z)=1-\frac{|z-z_0|^{2}}{4}=1-\frac{|z|^{2}-\bar{z}z_0-\bar{z_0}z+|z_0|^{2}}{4}$, which belongs to $S=Span\{1,z,\bar{z}, |z|^{2} \}$,  is the appropriate function for Remark \ref{RC}.

The two above results holds true for  "uniformly convergence" instead of "pointwise  convergence".

}
\end{exam}

\begin{rem} \label{ex} \em{From our theorems, one can obtain the Korovkin-type results of \cite{9} and \cite{14} (with respect to both "uniformly convergence" and "pointwise  convergence"), which are generalizations of  Korovkin's second theorem on convergence of a  sequence of positive linear maps  for the space of real-valued continuous 2$\pi$-periodic functions
on $\Bbb R$.}
\end{rem}


\begin{thebibliography}{BNT}

\bibitem{Alt} F. Altomare, \textit{Korovkin-type theorems and approximation by positive linear
operators},  Surv. Approx. Theory. {\bf 5} (2010), 92-164.

\bibitem{8} F. Altomare and M. Campiti, \textit{Korovkin-Type Approximation Theory and its Applications}, de
Gruyter Studies in Mathematics, 17, Walter de Gruyter Co., Berlin, 1994.

\bibitem{AF} J. Araujo and J. J. Font,  \textit{Linear isometries between subspaces of
continuous functions}, Trans. Amer. Math. Soc. {\bf 349} (1997), 413-428.


%\bibitem{AF}  J. Araujo and J.J. Font, \textit{On $\breve{S}$ilov
%boundaries for subspaces of continuous functions}, Topology Appl.
%77 (1997)  79-85.

\bibitem{Bi} E. Bishop and K. de Leeuw, \textit{The representation of linear functionals by measures on sets of extreme points}, Ann. Inst. Fourier (Grenoble) \textbf{9} (1959), 305-331.

\bibitem{Br} A. Browder,  \textit{Introduction to Function Algebras}, W. A. Benjamin, New York-Amsterdam, 1969.


%\bibitem{D}  H.G. Dales,   \textit{Boundaries and peak points for
%Banach function algebras}, Proc. London Math. Soc.  22 (1971)
%121-136.

%\bibitem{E} A.J. Ellis, \textit{ Real characterizations of function algebras amongst
%function spaces}, Bull. London Math. Soc. 22  (1990) 381-385.


%\bibitem{F} R.J. Fleming and J.E. Jamison, \textit{Isometries on Banach Spaces:
%Function Spaces}, Chapman  Hall/CRC Monogr. Surv. Pure Appl.
% Math., 129, Chapman  Hall/CRC, Boca Raton, 2003.



%\bibitem{FS1} J.J. Font and M. Sanchis,  \textit{Bilinear isometries on subspaces
% of continuous functions}, Math. Nachr. 283  (2010) 568-572.

\bibitem{46} K. Donner, \textit{Extension of Positive Operators and Korovkin Theorems}, Lecture Notes in Math.
904, Springer-Verlag, Berlin, 1982.




%\bibitem{HLLMTY} O. Hatori, S. Lambert,  A. Lutman, T. Miura, T. Tonev and R.Yates,
%\textit{Spectral preservers in commutative Banach algebras},
%Contemp. Math. 547 (2011)  103-123.


%\bibitem{Hat-Char} O.  Hatori,  T. Miura and H. Takagi,  \textit{Characterizations of
%isometric isomorphisms between uniform  algebras via nonlinear
%range-preserving properties}, Proc. Amer. Math. Soc.  134 (2006)
%2923-2930.

\bibitem{HO}  T. Hachiro and T.  Okayasu, \textit{Some theorems of Korovkin type}, Studia Math. {\bf 155} (2003), no. 2, 131-143.


%\bibitem{Hol} W. Holszty\'{n}ski,  \textit{Continuous mappings induced by isometries of
%spaces of continuous functions}, Studia Math. 26  (1966) 133-136.

%\bibitem{H} M. Hosseini, \textit{Real-linear isometries on spaces of functions of bounded
%variation},  Results Math. (2015), to appear

\bibitem{Ko} P. P.
Korovkin,
\textit{On convergence of linear positive operators in the space of continuous functions},
 Doklady Akad. Nauk SSSR (N.S.) {\bf 90} (1953), 961-964 (in Russian).


\bibitem{Miu}
 H. Koshimizu, T. Miura, H. Takagi and  S. E. Takahasi,
{\it Real-linear isometries between subspaces of continuous
functions}, J. Math. Anal. Appl. {\bf 413} (2014), 229-241.





\bibitem{M} S. Mazur and S. Ulam, \textit{Sur les transformations
isom\'{e}triques d'espaces vectoriels norm\'{e}s}, C. R. Math.
Acad. Sci. Paris {\bf 194} (1932), 946-948.

\bibitem{9}
E. N. Morozov, \textit{Convergence of a sequence of positive linear operators in the space of continuous 2$\pi$-periodic functions of two variables}, Kalinin. Gos. Ped. Inst. Uchen. Zap. {\bf 26} (1958), 129-142 (in Russian).









%\bibitem{Mu} F. J. Munoz-Delgado and  D.  C\'{a}rdenas-Morales, \textit{Almost convexity and quantitative Korovkin type results}, Appl. Math. Lett. {\bf 11} (1998), no. 4, 105-108.

\bibitem{P1} R. R. Phelps, \textit{Lectures on Choquet's Theorem}, 2nd ed., Lecture Notes in Math. 1757, Springer, Berlin, 2001.

\bibitem{P} R. R. Phelps, \textit{The range of Tf for certain linear operators $T$},  Proc. Amer. Math. Soc. {\bf 16} (1965), 381-382.

%\bibitem{Mi}  T.  Miura,  {\it Real-linear
%isometries between function algebras}, Cent. Eur. J. Math. 9
%(2011) 778-788.

\bibitem{Sche} E. Scheffold, \textit{Uber die punktweise konvergenz von operatoren in $C(X)$},  Rev. Acad. Ci.
Zaragoza {\bf 28} (1973), 5-12.



%\bibitem{RR2}  N.V. Rao and A.K. Roy,  \textit{ Multiplicatively
%spectrum-preserving maps of function algebras II},  Proc. Edinb.
%Math. Soc. 48 (2005)  219-229.

\bibitem{14}
V. I. Volkov, \textit{Conditions for convergence of a sequence of positive linear operators in the space of continuous functions of two variables}, Kalinin. Gos. Ped. Inst. Uchen. Zap. {\bf 26} (1958), 27-40 (in Russian).



\end{thebibliography}
\end{document}